\documentclass[12pt,leqno]{article}
\usepackage{amsfonts}
\usepackage{amsmath, amsthm, amsfonts, amssymb, color}
\usepackage{mathrsfs}
\newtheorem{thm}{Theorem}[section]

\newtheorem{lem}[thm]{Lemma}

\theoremstyle{definition}

\newcommand{\scr}[1]{\mathscr #1}
\definecolor{wco}{rgb}{0.5,0.2,0.3}

\numberwithin{equation}{section} \theoremstyle{remark}

\renewcommand{\bar}{\overline}
\newcommand{\ua}{\uparrow}

\renewcommand{\tilde}{\widetilde}
\title{{\bf Hypercontractivity  for   Functional Stochastic Partial Differential Equations}\footnote{Supported in
 part by  Lab. Math. Com. Sys., the 985 project and  NNSFC(11131003, 11431014, 11401592).}
}
\author{
{\bf  Jianhai Bao$^{b)}$,  Feng-Yu Wang$^{a),c)}$,  Chenggui Yuan$^{c)}$}\\
\footnotesize{$^{a)}$School of Mathematical Sciences, Beijing Normal
University, Beijing 100875, China}\\
\footnotesize{$^{b)}$School of Mathematics and Statistics, Central
South
University, Changsha 410083, China}\\
 \footnotesize{$^{c)}$Department of Mathematics,
Swansea University, Singleton Park, SA2 8PP, UK}\\
\footnotesize{jianhaibao13@gmail.com, wangfy@bnu.edu.cn,
F.-Y.Wang@swansea.ac.uk, C.Yuan@swansea.ac.uk}}
\begin{document}
\def\R{\mathbb R}  \def\ff{\frac} \def\ss{\sqrt} \def\B{\mathbf
B}
\def\N{\mathbb N} \def\kk{\kappa} \def\m{{\bf m}}
\def\dd{\delta} \def\DD{\Dd} \def\vv{\varepsilon} \def\rr{\rho}
\def\<{\langle} \def\>{\rangle} \def\GG{\Gamma} \def\gg{\gamma}
  \def\nn{\nabla} \def\pp{\partial} \def\EE{\scr E}
\def\d{\text{\rm{d}}} \def\bb{\beta} \def\aa{\alpha} \def\D{\scr D}
  \def\si{\sigma} \def\ess{\text{\rm{ess}}}
\def\beg{\begin} \def\beq{\begin{equation}}  \def\F{\scr F}
\def\Ric{\text{\rm{Ric}}} \def\Hess{\text{\rm{Hess}}}
\def\e{\text{\rm{e}}} \def\ua{\underline a} \def\OO{\Omega}  \def\oo{\omega}
 \def\tt{\tilde} \def\Ric{\text{\rm{Ric}}}
\def\cut{\text{\rm{cut}}} \def\P{\mathbb P} \def\ifn{I_n(f^{\bigotimes n})}
\def\C{\scr C}      \def\aaa{\mathbf{r}}     \def\r{r}
\def\gap{\text{\rm{gap}}} \def\prr{\pi_{{\bf m},\varrho}}  \def\r{\mathbf r}
\def\Z{\mathbb Z} \def\vrr{\varrho} \def\l{\lambda}
\def\L{\scr L}\def\Tt{\tt} \def\TT{\tt}\def\II{\mathbb I}
\def\i{{\rm in}}\def\Sect{{\rm Sect}}\def\E{\mathbb E} \def\H{\mathbb H}
\def\M{\scr M}\def\Q{\mathbb Q} \def\texto{\text{o}} \def\LL{\Lambda}
\def\Rank{{\rm Rank}} \def\B{\scr B} \def\i{{\rm i}} \def\HR{\hat{\R}^d}
\def\to{\rightarrow}\def\l{\ell}\def\ll{\lambda}
\def\8{\infty}\def\ee{\epsilon}\def\V{\mathbb{V}}
\def\DD{\Delta} \def\Y{\mathbb{Y}} \def\lf{\lfloor}
\def\rf{\rfloor}\def\3{\triangle}\def\H{\mathbb{H}}\def\S{\mathbb{S}}\def\1{\lesssim}
\def\va{\varphi}

\maketitle

\begin{abstract}
Explicitly sufficient conditions on the hypercontractivity are
presented for two classes of functional stochastic partial
differential equations driven by, respectively,  non-degenerate and
degenerate Gaussian noises. Consequently, these conditions   imply
that the associated Markov semigroup is $L^2$-compact and
exponentially convergent to the stationary distribution  in entropy,
variance and total variational norm.
 As the log-Sobolev inequality is invalid under the present framework, we apply a  criterion  presented  in the recent paper \cite{Wang14} using
  Harnack inequality, coupling property and  Gaussian concentration property of  the stationary distribution.   To verify the concentration property,
  we prove a Fernique type inequality for infinite-dimensional Gaussian processes  which might be interesting by itself. \end{abstract}


\section{Introduction}

The hypercontractivity was  introduced in 1973 by Nelson \cite{N}
for the Ornstein-Ulenbeck semigroup. As applications,  it implies
the exponential convergence of the Markov semigroup in entropy (and
hence, also in variance) to the associated  stationary distribution,
and it also implies the $L^2$-compactness of the semigroup subject
to the existence of a density with respect to the stationary
distribution, see \cite{Wang14} for more details. In the setting of
symmetric Markov processes, Gross  \cite{Gross}  proved that
 the hypercontractivity of the semigroup is equivalent to the log-Sobolev inequality for the associated Dirichlet form. This leads to an extensive study of the log-Sobolev inequality.

 However, as explained in \cite{BWY} 
 the log-Sobolev inequality does not hold
for the segment solution to a stochastic delay differential equation (SDDE). As the segment solution is a process on a functional space, the equation
 is also called a   functional stochastic differential equation (FSDE).  In this case, an   efficient tool to prove the hypercontractivity is the
 dimension-free Harnack inequality introduced in \cite{W97}, where diffusion semigroups on Riemannian manifolds are concerned.
By using the coupling by change of measures, this type Harnack inequality has been  established for various 
 stochastic equations, see the recent monograph  \cite{W13} and references within.  The aim of the present paper is to prove
 the hypercontractivity for   functional stochastic partial differential equations (FSPDEs) in Hilbert spaces. We will consider
 non-degenerate noise  and degenerate noise, respectively, so that   the corresponding results   derived in  \cite{BWY}  for finite-dimensional FSDEs as well as   in \cite{Wang14}  for
degenerate   SPDEs  are extended.

In the recent paper \cite{Wang14}, the second named author developed
a general criterion on the hypercontractivity by using the Harnack
inequality of the semigroup, the concentration property of the
underlying probability measure, and the coupling property. In
general, let $P_t$ be a Markov semigroup on $L^2(\mu)$ for a
probability space $(E,\F,\mu)$ such that $\mu$ is $P_t$-invariant.
By definition, $P_t$ is hypercontractive if $\|P_t\|_{2\to 4} =1$
holds for large enough $t>0$, where $\|\cdot\|_{2\to 4}$ is the
operator norm from $L^2(\mu)$ to $L^4(\mu)$.  For any $(x,y)\in
E\times E$, a process $(X_t,Y_t)$ on $E\times E$ is called a
coupling for the Markov semigroup with initial point $(x,y)$ if
$$P_tf(x)= \E f(X_t),\ \ P_t f(y)= \E f(Y_t),\ \ t\ge 0, f\in \B_b(E),$$ where $\B_b(E)$ stands for  the set of all bounded measurable functions
defined
on $E$.

 The general criterion  due to Wang \cite{Wang14} is stated  as follows.

\beg{thm}[\cite{W13}]\label{T1.1} Assume that   the following three conditions hold for some
measurable functions $\rr:E\times E\mapsto(0,\8)$ and
$\phi:[0,\8)\mapsto(0,\8)$ such that $\lim_{t\to\8}\phi(t)=0:$
\begin{enumerate}
\item[{\rm (i)}] {\bf (Harnack Inequality)} There exist constants $t_0,c_0>0$ such that
\begin{equation*}
(P_{t_0}f(\xi))^2\le
(P_{t_0}f^2(\eta))\e^{c_0\rr(\xi,\eta)^2},~~~f\in\B_b(E),~\xi,\eta\in
E;
\end{equation*}
\item[{\rm (ii)}] {\bf (Coupling Property)}  For any $(\xi,\eta)\in E\times E$, there exists a coupling
$(X_t,Y_t)$ for the Markov semigroup $P_t$ such that
\begin{equation*}
\rr(X_t,Y_t)\le\phi(t)\rr(\xi,\eta),~~~t\ge0;
\end{equation*}
\item[{\rm (iii)}] {\bf (Concentration Property)} There exists   $\vv>0$ such that
$(\mu\times\mu)(\e^{\vv\rr(\cdot,\cdot)^2})<\8$.
\end{enumerate}Then $P_t$ is hypercontractive  and   compact in $L^2(\mu)$ for large enough $t>0$,  and
\begin{equation}\label{EXX} \beg{split}
&\mu((P_tf)\log P_t f)\le c\e^{-\aa t}\mu(f\log f),\ \ t\ge 0,   f \ge 0, \mu(f)=1;\\
&\|P_t-\mu\|_2^2:= \sup_{\mu(f^2)\le 1}\mu\big((P_t
f-\mu(f))^2\big)\le c\e^{-\aa t},\ \ t\ge 0
\end{split}\end{equation} hold for some constants $c,\aa>0.$
 \end{thm}

 We will apply the previous criterion to non-degenerate and degenerate FSPDEs, respectively. To state our main results,
we first introduce some notation.

For  two separable Hilbert spaces $\H_1,\H_2$,   let $\L(\H_1,\H_2)$
(respectively,  $\L_{HS}(\H_1,\H_2)$) be the set of all bounded
(respectively, Hilbert-Schmidt) linear operators from $\H_1$ to
$\H_2$. We will use $|\cdot|$ and $\<\cdot,\cdot\>$ to denote the
norm and the inner product on a Hilbert space, and let $\|\cdot\|$
and $\|\cdot\|_{HS}$ stand for the operator norm and the
Hilbert-Schmidt norm for a linear operator. Below we introduce our
main results for non-degenerate FSPDEs and degenerate FSPDEs,
respectively.

\subsection{ Non-Degenerate FSPDEs}

Let $\H$ be a separable Hilbert space.
\smallskip
For a fixed constant $r_0>0$, let
$\mathscr{C}=C([-r_0,0];\mathbb{H})$ be  equipped with the uniform
norm $\|f\|_\infty:=\sup_{-r_0\leq\theta\leq0}|f(\theta)|$. For
$t\geq0$ and   $h\in C([-r_0,\infty);\mathbb{H})$, let
$h_t\in\mathscr{C}$ be  such that
$h_t(\theta)=h(t+\theta),\theta\in[-r_0,0]$.

\smallskip

Let $W(t)$ be a cylindrical Wiener process on $\H$ under a complete
filtered probability space $(\OO,\F,\{\F_t\}_{t\ge 0}, \P)$; that
is,
$$W(t)=\sum_{i=1}^\infty B_i(t)e_i,\ \ t\ge 0$$ for an orthonormal basis $\{e_i\}_{i\ge 1}$ on $\H$ and a sequence of independent one-dimensional
 Wiener processes $\{B_i(t)\}_{i\ge 1} $
on $(\OO,\F,\{\F_t\}_{t\ge 0}, \P)$.

Consider the following  FSPDE on $\H$:
\begin{equation}\label{*}
\d X(t)=\{AX(t)+b(X_t)\}\d t+\si\d W(t),~~t>0,~~X_0=\xi\in\C,
\end{equation}
where $(A,\D(A))$ is a densely defined closed operator on $\H$
generating a $C_0$-semigroup $\e^{tA}$, $b:\C\mapsto\H$ is
measurable, $(\si, \D(\si))$   is a densely defined  linear operator
on $\H$.   We assume that $A, b$ and $\si$ satisfy the following
conditions.

\begin{enumerate}
\item[{\bf (A1)}] ($-A, \mathscr{D}(A))$  is   self-adjoint  with
  discrete spectrum
 $0<\ll_1\le \ll_2\le \cdots $ counting  multiplicities such that $\ll_i\uparrow\infty$. Moreover,  there exists a constant
 $\dd\in (0,1)$ such that,  for every $t>0$, $\e^{-t(-A)^{1-\dd}}\si$ extends to a unique Hilbert-Schmidt operator on $\H$ which is denoted again by $\e^{-t(-A)^{1-\dd}}\si$ and satisfies
 \beq\label{JJ} \int_0^1 \|\e^{-t(-A)^{1-\dd}} \si\|_{HS}^2 \d t<\infty.\end{equation}
\item[{\bf (A2)}] There exists a constant $L>0$ such that $|b(\xi)-b(\eta)|\le L\|\xi-\eta\|_\8,~\xi,\eta\in\C$.
\item[{\bf (A3)}] $\si$ is invertible, i.e.,  there exists  $\si^{-1}\in \L(\H, \H)$ such that $\si^{-1}\H\subset \D(\si)$ and $\si\si^{-1}=I$, the identity operator.
\end{enumerate}

We first observe that assumptions {\bf (A1)} and {\bf (A2)} imply the existence and  uniqueness of continuous mild solutions to \eqref{*}; that is,
  for any $\F_0$-measurable random variable $X_0=\xi\in\C$,  there exists a unique continuous adapted process $\{X(t)\}_{t\ge r_0}$ on $\H$ such that $\P$-a.s.
\begin{equation}\label{a1}
X(t)=\e^{t A}\xi(0)+\int_0^t\e^{(t-s)A}b(X_s)\d
s+\int_0^t\e^{(t-s)A}\si\d W(s),\ \ t\ge 0.
\end{equation} To this end, it suffices to show that
 \eqref{JJ} implies
\beq\label{JJ0} \int_0^1\|\e^{tA} \si\|_{HS}^{2(1+\vv)}\d t<\infty
\end{equation} for some $\vv>0$, see,    for instance, \cite[Theorem
4.1.3]{W13}.
 To prove \eqref{JJ0}, we
   reformulate condition \eqref{JJ} using the eigenbasis $\{e_i\}_{i\ge 1} $ of $A$, i.e., $\{e_i\}_{i\ge 1}$ is
   an orthonormal basis of $\H$ such that $Ae_i=-\ll_i e_i, i\ge 1.$
By noting that
$$\|\e^{-t(-A)^{1-\dd}} \si\|_{HS}^2=\|(\e^{-t(-A)^{1-\dd}} \si)^*\|_{HS}^2= \sum_{j=1}^\infty |(\e^{-t(-A)^{1-\dd}}\si)^*e_j|^2 = \sum_{j=1}^\infty \e^{-2 \ll_j^{1-\dd}t} |\si^* e_j|^2,$$
\eqref{JJ} is equivalent to \beq\label{JJ'} \sum_{j=1}^\infty
\ff{|\si^* e_j|^2}{\ll_j^{1-\dd}}<\infty.\end{equation} This implies
that  $\mu_j:= \ff{|\si^* e_j|^2}{\ll_j^{1-\dd}} (j\ge 1)$ gives
rise to a finite measure on $\mathbb N$,  so that by  H\"older's
inequality,
\beg{equation*}\beg{split}  &\int_0^1\|\e^{tA} \si\|_{HS}^{2(1+\vv)}\d t  = \int_0^1 \Big(\sum_{j=1}^\infty \e^{-2 \ll_j t} |\si^* e_j|^2\Big)^{1+\vv}\d t \\
&= \int_0^1 \Big(\sum_{j=1}^\infty \mu_j  \e^{-2 \ll_j t}
\ll_j^{1-\dd}  \Big)^{1+\vv}\d t
  \le C \int_0^1 \Big(\sum_{j=1}^\infty\mu_j \ll_j^{(1+\vv)(1-\dd)} \e^{-2(1+\vv)\ll_j t}\Big) \d t \\
&\le C \sum_{j=1}^\infty \ff{ |\si^* e_j|^2} { \ll_j^{1- \vv (1-\dd)}} <\infty,\ \ \vv\le \ff\dd{1-\dd},\end{split}\end{equation*}
where $C:= (\sum_{i=1}^\infty \mu_i)^\vv.$ Thus, \eqref{JJ} implies \eqref{JJ0} for $\vv\in (0, \ff\dd{1-\dd}].$

To emphasize the initial datum $X_0=\xi\in\C$, we denote the
solution and the segment solution by $\{X^\xi(t)\}_{t\ge-r_0}$ and
$\{X_t^\xi\}_{t\ge0}$, respectively.  Then the Markov semigroup for
the segment solution is defined as
 \begin{equation}\label{w8}
P_tf(\xi)=\E f(X_t^\xi),~f\in\B_b(\C),~\xi\in\C.
\end{equation}

  We are ready to state the main result in this part.

\begin{thm}\label{T1.2}
  Let {\bf (A1)}-{\bf (A3)} hold.  If $\ll:=\sup_{s\in (0,\ll_1]} \big(s-L\e^{sr_0}\big)>0$,
then  the following assertions hold. \beg{enumerate}
\item[$(1)$] $P_t$ has a unique invariant
probability measure $\mu$ such that $\mu(\e^{\vv\|\cdot\|_\8^2})<\8$
for some $\vv>0$.
\item[$(2)$] $P_t$ is hypercontractive  and   compact in $L^2(\mu)$ for large enough $t>0$,   and $\eqref{EXX}$ holds
for some constants $c,\aa>0$.
\item[$(3)$] For any $t_0>r_0$, there exists a constant  $  c>0$ such that
$$\|\mu_t^\xi-\mu_t^\eta\|_{\mbox{var}}\le  c\|\xi-\eta\|_\infty \e^{-\ll t},\ \ t\ge t_0,$$ where  $\|\cdot\|_{\mbox{var}}$ is the total variational norm and  $\mu_t^\xi$ stands for  the law of $X_t^\xi$ for   $(t,\xi)\in
 [0,\infty)\times\C$.
\end{enumerate}
\end{thm}

\smallskip

To illustrate the above result, we present below an example where $\H=L^2(D;\d x)$ for a bounded domain in $\R^d$.

\paragraph{Example 1.1.} For a bounded domain $D\subset \R^d$, let $\H=L^2(D;\d x)$ and $A=-(-\DD)^\aa$, where $\DD$ is the Dirichlet  Laplacian on $D$ and 
$\aa>\ff d 2$ is a constant. Let $\si=I$ be the identity operator on $\H$, and $b(\xi)=  L \int_{-r_0}^0 \xi(r)\nu(\d r)$ for a signed measure $\nu$ on $[-r_0,0]$  with total variation $1;$ or $b(\xi)= \sup_{r\in [-r_0,0]}\<\xi(r), g(r)\>$ for some measurable $g: [-r_0,0]\to\H$ with $ \|g\|_\infty\le L.$
Then assertions in Theorem \ref{T1.2} hold provided
$$\ll:=\sup_{s\in (0, (d\pi^2)^\aa R(D)^{-2\aa})} (s-L \e^{s r_0})>0,$$where $R(D)$ is the diameter of $D$. 
\beg{proof} Since $A=-(-\DD)^\aa$, it is well known that the eigenvalues $\{\ll_i\}_{i\ge 1}$ of $A$ satisfy $\ll_i\ge ci^{\ff{2\aa}d} (i\ge 1)$ for some constant $c>0$. So, for $\aa>\ff d 2$ assumptions {\bf(A1)}-{\bf (A3)} hold for the above choices of $\H, A, \si$ and $b$. By Theorem \ref{T1.2}, it remains to prove
$\ll_1\ge \ff{(d\pi^2)^{\aa}}{R(D)^{2\aa}}$. Letting $\bar \ll_1$ be the first  eigenvalue of $-\DD$, by the definition of $A$  this is equivalent to $\bar\ll_1\ge \ff{d\pi^{2}}{R(D)^{2}}$. As $D$ is covered by a cube of edge length $R(D)$, by the domain-monotonicity  and the shift-invariance of the first Dirichlet eigenvalue of $-\DD$,
$\bar\ll_1$ is bounded below by the first Dirichlet eigenvalue of $-\DD$ on the cube $[0,R(D)]^d$, which is equal to $\ff{d\pi^{2}}{R(D)^{2}}$ with eigenfunction
$u(x):=\prod_{i=1}^d \sin\big(\ff{\pi x_i}{R(D)}\big). $ Then  the proof is finished. \end{proof}

\subsection{Degenerate FSPDEs}

Let $\H=\H_1\times\H_2$ for two separable Hilbert spaces $\H_1$ and
$\H_2$, and let $\C=C([-r_0,0];\mathbb{H})$ as in Subsection 1.1.
Consider the following degenerate FSPDE on $\H$:
\beq\label{E1} \beg{cases} \d X(t)= \{A_1X(t)+  BY(t)\}\d t,\\
 \d Y(t)= \{A_2Y(t)+ b(X_t,Y_t)\}\d t +\si \d W(t),
\end{cases}
\end{equation} where  $(A_i,\D(A_i))$ is a densely defined closed linear operator  on $\H_i$  generating a $C_0$-semigroup $\e^{t A_i}$  ($i=1,2$), $B\in \L(\H_2, \H_1)$,
$b: \C\mapsto \H_2$ is measurable, $(\si, \D(\si))$ is a densely
defined closed operator on $\H_2$, and  $W(t)$ is the cylindrical
Wiener process on $\H_2$. Corresponding to {\bf (A1)}-{\bf (A3)} in
the non-degenerate case, we make the following assumptions (see
\cite{Wang14} for the case without delay, i.e., $b(X_t,Y_t)$ depends
only on $X(t)$ and $Y(t)$).

\smallskip

\beg{enumerate}
\item[{\bf (B1)}] ($-A_2, \mathscr{D}(A_2))$  is   self-adjoint  with
  discrete spectrum
 $0<\ll_1\le \ll_2\le \cdots $ counting  multiplicities such that $\ll_i\uparrow\infty$,   $\si$ is invertible, and
$$ \int_0^1 \|\e^{-t(-A_2)^{1-\dd_0}} \si\|_{HS}^2 \d t<\infty$$  holds for some constant $\dd_0\in (0,1)$.
\item[{\bf (B2)}] There exist constants $K_1,K_2> 0$ such that
\begin{equation*}
|b(\xi_1,\eta_1)-b(\xi_2,\eta_2)|\le
K_1\|\xi_1-\eta_1\|_\8+K_2\|\xi_2-\eta_2\|_\8,~~(\xi_1,\eta_1),~(\xi_2,\eta_2)\in\C.
\end{equation*}
\item[{\bf (B3)}] $ A_1 \le \dd-\ll_1 $   for some constant  $ \dd\ge 0$; i.e., $\<A_1 x,x\>\le (\dd-\ll_1)|x|^2$ holds for all $x\in \D(A_1)$.
\item[{\bf (B4)}]  There exists $A_0\in \L(\H_1, \H_1)$ such that $B\e^{tA_2}= \e^{tA_1}\e^{tA_0} B$ holds for $t\ge 0$,  and
\begin{equation*}
Q_t:=\int_0^t\e^{sA_0}BB^*\e^{sA_0^*}\d s,~~~t\ge0
\end{equation*}
is invertible on $ \H_1$.
\end{enumerate} Obviously, when $\H_1=\H_2, \si=B=I$  and $ A_1=A_2$ with discrete spectrum $\{-\ll_i\}_{i\ge 1}$ such
that $\sum_{i=1}^\infty \ff{1}{\ll_i^{1-\dd}}<\infty$ holds for some
constant $\dd\in (0,1),$ then assumptions {\bf (B1)}, {\bf (B3)} and
{\bf (B4)} hold. See \cite{Wang14} for more examples, where $\H_2$
might be a subspace of $\H_1.$

Similarly to  the case without delay considered in  \cite{Wang14},
assumptions {\bf (B3)} and {\bf (B4)} will be used to prove the
Harnack inequality. Moreover, as explained in Subsection 1.1 for the
non-degenerate case, from   \cite[Theorem 4.1.3]{W13}  we conclude
that assumptions {\bf (B1)} and {\bf (B2)} imply the existence,
uniqueness and non-explosion of the continuous mild solution
$(X^{\xi,\eta}(t), Y^{\xi,\eta}(t))$ for any initial point $(\xi,
\eta)\in \C.$ Let $P_t$ be the  Markov semigroup generated by the
segment solution. We have
$$
P_tf(\xi,\eta)=\E\big[
f(X_t^{\xi,\eta},Y_t^{\xi,\eta})\big],~~f\in\B_b(\C),~~(\xi,\eta)\in\C,~~t\ge0.$$

\begin{thm}\label{T1.3} Assume
{\bf  (B1)}-{\bf (B4)}. If
\begin{equation}\label{c1}
\ll':=\ff1
2\Big(\dd+K_2+\ss{(K_2-\dd)^2+4K_1\|B\|}\Big)<\sup_{s\in (0,\ll_1]}  s\e^{-sr_0},
\end{equation}
then all assertions in Theorem $\ref{T1.1}$ hold with
$\ll:=\sup_{s\in (0,\ll_1]} \big(s-\e^{s r_0}\ll'\big)$.
\end{thm}

\paragraph{Examples 1.2.} Let $\H_1=\H_2=L^2(D;\d x)$ and $\si=B=I, A_1=A_2=-(-\DD)^\aa$ for some $\aa>\ff d 2$ as in Example 1.1. Then   assumptions {\bf (B1)}, {\bf (B3)} and
{\bf (B4)} hold. See \cite{Wang14} for more examples, where $\H_2$
might be a subspace of $\H_1.$ To verify {\bf (B2)} we take, for instances, 
$$b(\xi,\eta)= K_1\int_{-r_0}^0 \xi\d\nu_1+K_2\int_{-r_0}\eta\d\nu_2$$ for some signed measures $\nu_1,\nu_2$ on $[-r_0,0]$ with total variations not larger than $1$; or simply 
$b(\xi,\eta)= \|K_1\xi+K_2\eta\|_\infty$ where interactions exist between $\xi$ and $\eta$. 

\

The remainder of this paper is organized as follows. In Section 2 we
present a Fernique type inequality for infinite-dimensional Gaussian
processes, which will be used to prove the concentration condition
required in Theorem \ref{T1.1}(3).     Theorems \ref{T1.2} and
\ref{T1.3} are proved in Sections 3 and 4, respectively.

\section{ Infinite-dimensional Fernique's inequality}

 In \cite{F},  Fernique introduced an inequality for the distribution of the maximum of Gaussian processes. To prove the exponential integrability
  of $\|X_t\|_\infty$ for FSPDEs, one needs an infinite-dimensional version of this inequality. However, as the dimension goes to infinity,
 existing Fernique's inequality for multi-dimensional Gaussian processes becomes invalid.
 So, we modify the   inequality so that it holds also in infinite-dimensions.
 To this end, we first recall the inequality for one-dimensional Gaussian processes (see, e.g., \cite[page 49]{Berman} for the multi-dimensional  case).

\beg{lem}[Fernique's inequality] \label{LL1} Let $\{\gg(t)\}_{t\in
[0,1]}$ be a continuous Gaussian process on $\R$ with zero mean and
$\GG= \sup_{t\in [0,1]} (\E \gg(t)^2)^{\ff 1 2}<\infty$. Let
$$\phi(r):= \sup_{s,t\in [0,1], |s-t|\le r} \big(\E|\gg(s)-\gg(t)|^2\big)^{\ff 1 2},\ \ r\in [0,1].$$
If $\theta:=\int_1^\infty \phi(\e^{-s^2})\d s<\infty$, then
$$\P\Big(\max_{t\in[0,1]} |\gg(t)|\ge r\big(\GG  + \big(2+\ss 2\big)\theta\big)\Big)\le \ff{5\e} 2 \int_r^\infty \e^{-\ff 12 s^2} \d s,\ \ r\ge \ss{5}.$$
\end{lem}

Now, we call a process $\{\gg(t)\}_{t\in [0,1]}$  on the Hilbert
space $\H$  a cylindrical continuous Gaussian  process, if, for an
orthonormal basis $\{e_i\}_{i\ge 1}$, every  one-dimensional process
$\gg_i(t):= \<\gg(t),e_i\>$ is a continuous  Gaussian process. For a
cylindrical continuous Gaussian process $\gg(t)$ with zero mean, let
\beg{equation*}\beg{split}
&\phi_i(r)= \sup_{s,t\in [0,1], |s-t|\le r} \big(\E |\gg_i(t)-\gg_i(s)|^2\big)^{\ff 1 2},\ \ r\in [0,1],\\
&\GG_i= \sup_{t\in [0,1]} (\E \gg_i(t)^2)^{\ff 1 2},\ \ \dd_i=
\GG_i+ \big(2+\ss 2\big)\int_1^\infty \phi_i(\e^{-s^2})\d s,\ \ \
i\ge 1.\end{split}\end{equation*}

\beg{thm}\label{T2.1} Let $\gg(t)$ be a cylindrical continuous
Gaussian process on $\H$ with zero mean such that \beq\label{TH}
\theta:= \sum_{i=1}^\infty \dd_i^2
\log(\e+\dd_i^{-1})<\infty.\end{equation} Then, for any positive
constant
 $\ll < \min_{i\ge 1}  \ff{\log(\e +\dd_i^{-1})}{2\theta},$ there exists a constant $c>0$ such that
\beq\label{TH2} \P\Big(\max_{t\in [0,1]} |\gg(t)|  \ge r \Big)\le c\e^{-\ll r^2},\ \ r\ge 0.\end{equation}\end{thm}

 \beg{proof}  Let $\tt\ll= \min_{i\ge 1}  \ff{\log(\e +\dd_i^{-1})}{2\theta}.$  Obviously, \eqref{TH} implies $\lim_{i\to\infty}\dd_i=0$ so that $\tt\ll>0.$ For any $\ll\in (0, \tt\ll)$, it suffices to prove \eqref{TH2} for some constant $c>0$ and large enough $r>0$.
 Below, we assume that
 \beq\label{RR} r^2 \ge \ff{5\theta\tt\ll}{\tt\ll-\ll}.\end{equation}
 In this case, $$r_i := \Big(\ff{r^2\log(\e +\dd_i^{-1})}\theta \Big)^{\ff 1 2} \ge \ff r{\ss \theta} \ge \ss 5,$$ so that Lemma \ref{LL1} implies
 $$\P\Big(\max_{t\in[0,1]} |\gg_i(t)|  \ge r_i \dd_i \Big) \le \ff{5 \e} 2 \int_{r_i}^\infty \e^{-\ff 1 2 s^2}\d s \le c_1 \e^{-\ff 1 2 r_i^2},\ \ i\ge 1$$
 for some constant $c_1>0.$ Then
 \beg{equation}\label{NN} \beg{split} & \P\Big(\max_{t\in [0,1]} |\gg(t)|  \ge r \Big)\le \P\Big(\sum_{i=1}^\infty \max_{t\in [0,1]} |\gg_i(t)|^2  \ge r^2 \Big) \\
 &\le \sum_{i=1}^\infty \P\Big( \max_{t\in [0,1]} |\gg_i(t)|^2  \ge \ff{r^2\dd_i^2\log(\e +\dd_i^{-1})}{\theta} \Big)= \sum_{i=1}^\infty \P\Big(\max_{t\in[0,1]} |\gg_i(t)|  \ge r_i \dd_i \Big)\\
 &\le c_1\sum_{i=1}^\infty \e^{-\ff 1 2 r_i^2} \le c_1\e^{-\ll r^2} \sum_{i=1}^\infty \exp\Big[ - r^2\Big(\ff{\log (\e+\dd_i^{-1})}{2\theta} -\ll\Big)\Big].\end{split}\end{equation}
 Since, by \eqref{RR} and the definition of $\tt\ll$, we have
 $$r^2\Big(\ff{\log (\e+\dd_i^{-1})}{2\theta} -\ll\Big)\ge \ff{r^2  \log (\e+\dd_i^{-1})}{2\theta}\Big(1-\ff\ll{\tt\ll}\Big) \ge \ff 5 2 \log (\e+\dd_i^{-1}),$$
 it follows from \eqref{TH} that
 $$ \sum_{i=1}^\infty \exp\Big[ - r^2\Big(\ff{\log (\e+\dd_i^{-1})}{2\theta} -\ll\Big)\Big]\le \sum_{i=1}^\infty \dd_i^{\ff 5 2} <\infty.$$
 Combining this with \eqref{NN}, we finish the proof.
 \end{proof}

\section{Proof of Theorem \ref{T1.2}}

We will verify conditions (i)-(iii) in Theorem \ref{T1.1}. Firstly,
according to \cite[Theorem 4.2.4]{W13}, assumptions {\bf (A1)}-{\bf
(A3)} implies that, for any $t_0>r_0$, there exists a constant
$c_0>0$ such that the following Harnack inequality holds:
\beq\label{HH1} \big(P_{t_0} f(\eta)\big)^2\le (P_{t_0} f^2(\xi)))
\e^{c_0\|\xi-\eta\|_\infty^2},\ \ \ \xi,\eta\in\C, f\in
\B_b(\C).\end{equation} That is, condition (i) holds for
$\rr(\xi,\eta):=\|\xi-\eta\|_\infty$.

To verify   (ii) and (iii), we will need the condition that $\ll:=
\sup_{s\in (0,\ll_1]} (s-L \e^{sr_0})>0.$ Without loss of
generality, we may and do assume that the maximum  is attained at
the point $\ll_1$; otherwise, in the following it suffices to
replace $\ll_1$ by $\ll_1'\in (0,\ll_1]$ which attains the maximum.
By {\bf (A1)}, {\bf (A2)}, and \eqref{a1}, one has
\begin{equation*}
\e^{\ll_1t
}|X^{\xi} (t)-X^\eta(t)| \le|\xi(0)-\eta(0)| +L\int_0^t\e^{\ll_1s}\|X_s^\xi-X_s^\eta\|_\8\d
s.
\end{equation*}
Then, we obtain that
\begin{equation}\label{eq5}
\begin{split}
\e^{\ll_1t
}\|X_t^\xi-X_t^\eta\|_\8&\le\e^{\ll_1r_0}\sup_{-r_0\le\theta\le0}(\e^{\ll_1(t+\theta)
}|X^\xi(t+\theta)-X^\eta(t+\theta)|)\\
&\le\e^{\ll_1r_0}\Big(\|\xi-\eta\|_\8+L\int_0^t\e^{\ll_1s}\|X_s^\xi-X_s^\eta\|_\8\d
s\Big).
\end{split}
\end{equation}
Thus, by Gronwall's inequality we derive that
\begin{equation}\label{eq2}
\|X_t^\xi -X_t^\eta\|_\8\le \e^{\ll_1r_0} \e^{-\ll
t}\|\xi-\eta\|_\8,~~t\ge0,~~\xi,\eta\in\C.
\end{equation} That is, condition (ii) holds.

To show condition (iii) in Theorem \ref{T1.1}, we  need to prove the
exponential integrability of the segment solution.

\begin{lem}\label{Fer} Assume {\bf (A1)} and {\bf (A2)}.  If $\ll>0$, then there exists an  $r>0$ such that \
\begin{equation}\label{a4}
 \sup_{t\ge 0} \E\e^{r\|X_t^\xi\|_\8^2}<\infty,\ \ \ \xi\in\C.
\end{equation}    \end{lem}

\begin{proof} (a) We first use Theorem \ref{T2.1} to prove
\beq\label{a5} \sup_{t\ge 0} \E \e^{\vv
\|Z_{t}\|_\infty^2}<\infty\end{equation}   for   some $\vv>0$, where
\beq\label{b2} Z_t(\theta):= \int_0^{(t+\theta)^+}
\e^{(t+\theta-s)A} \si\d W(s),\ \ t\ge 0,~ \theta\in
[-r_0,0].\end{equation} To this end, for fixed $t_0>0$ let
$$\gg(t)= \int_0^{(t_0-tr_0)^+}\e^{(t_0-tr_0-s)A}\si\d W(s),\ \ t\in [0,1].$$ Then,  \eqref{b2}  implies
\beq\label{BB2} \|Z_{t_0}\|^2_\infty = \sup_{t\in
[0,1]}|\gg(t)|^2.\end{equation} Letting $\{e_i\}_{i\ge 1}$ be the
eigenbasis of $A$, we have \beq\label{BB3} \gg_i(t):=\<\gg(t),e_i\>=
\int_0^{(t_0-tr_0)^+}\e^{-\ll_i(t_0-tr_0-s)}\<\si^* e_i, \d W(s)\>,\
\ t\in [0,1].\end{equation} Obviously, \beq\label{BB4}
\GG_i:=\sup_{t\in [0,1]} \big(\E \gg_i(t)^2\big)^{\ff 1 2} \le
|\si^* e_i| \bigg(\int_0^\infty \e^{-2\ll_i s}\d s\bigg)^{\ff 1 2} =
\ff{|\si^* e_i|}{\ss{2\ll_i}},\ \ i\ge 1.\end{equation} Moreover,
note that, for any $r\in (0,1)$, there exists a constant $c(r)>0$
such that $|\e^{-s}-\e^{-t}|\le c(r) |s-t|^r$ holds for all $s,t\ge
0.$ Then, \eqref{BB3}, implies that for any $0\le t'\le t\le 1$,
\beg{equation*}\beg{split} &\E |\gg_i(t)-\gg_i(t')|^2\\
&=  |\si^* e_i|^2\bigg( \int_0^{(t_0-tr_0)^+} \e^{-2\ll_i (t_0-tr_0-s)}\big(1-\e^{-\ll_i(t-t')r_0}\big)^2 \d s +   \int_{(t_0-tr_0)^+}^{(t_0-t'r_0)^+}
\e^{-2\ll_i (t_0-t'r_0-s)} \d s\bigg)\\
& \le  |\si^* e_i|^2\bigg(\ff{c(\ff\dd 4)^2 [r_0(t-t')]^{\ff\dd 2}}{2\ll_i^{1-\ff \dd 2}} +
\ff{c(\ff \dd 2)[2r_0(t-t')]^{\ff\dd 2}}{2\ll_i^{1-\ff \dd 2}}\bigg)\\
&=: \ff{c_1 (t-t')^{\ff\dd 2}|\si^* e_i|^2}{\ll_i^{1-\ff\dd 2}},\ \
\ i\ge 1,\end{split}\end{equation*} where the constant $c_1>0$ is
independent of $t,t',t_0$ and $i$. So,  by the definition of
$\phi_i$,
$$\phi_i(r)\le \ff{c_1^{1/2} r^{\ff \dd 4} |\si^* e_i|}{\ll_i^{\ff 1 2-\ff \dd 4}},\ \ r\in [0,1].$$ Combining this with \eqref{BB4}, we deduce from
the definition of $\dd_i$ that
$$\dd_i\le \ff{c_2|\si^* e_i|}{\ll_i^{\ff 1 2-\ff\dd 4}},\ \ i\ge 1 $$ holds for some constant $c_2>0$ independent of $t_0$.  This
and \eqref{JJ'} lead to \eqref{TH}. Therefore, according to Theorem
\ref{T2.1} and \eqref{BB2}, we prove  \eqref{a5}  for some constant
$\vv\in (0,1)$.

(b) Next, we prove \eqref{a4} for small $r>0$. By \eqref{eq2}, it
suffices to prove for $\xi\equiv0.$
 We simply denote $X(t)=X^0(t).$
It follows from {\bf (A1)}, {\bf (A2)}, and \eqref{a1}  that
\begin{equation*}
\e^{\ll_1t}|X(t)|
\le \int_0^t\e^{\ll_1s}\{c_0+L\|X_s\|_\8\}\d
s+\e^{\ll_1t}\Big|\int_0^t\e^{(t-s)A}\si\d W(s)\Big|,\ \ t\ge 0
\end{equation*}
holds for some constant $c_0>0.$ This  implies
\begin{equation*}
\begin{split}
\e^{\ll_1t}\|X_t\|_\8&\le\e^{\ll_1
r_0}\sup_{-r_0\le\theta\le0}(\e^{\ll_1(t+\theta)}|X(t+\theta)|)\\
&\le
c_1 \e^{\ll_1t}(1+\|Z_{t}\|_\infty)+L\e^{\ll_1
r_0}\int_0^t\e^{\ll_1s}\|X_s\|_\8\d s
\end{split}
\end{equation*}
for some constant  $c_1>0$, where $ Z_{t}$ is defined  in
\eqref{b2}. So, by  Gronwall's formula,
\begin{equation*}
\begin{split}
\|X_t\|_\8&\le
c_1(1+\|Z_{t}\|_\infty)+c_1L\e^{\ll_1
r_0}\e^{-\ll_1t}\int_0^t\Big\{\e^{\ll_1s}+\e^{\ll_1s}\|Z_{s}\|_\infty\Big\}\e^{L\e^{\ll_1
r_0}(t-s)}\d s\\
&\le
c_2(1+\|Z_{t}\|_\infty)+c_2\int_0^t\|Z_{s}\|_\infty\e^{-\ll(t-s)}\d s
\end{split}
\end{equation*}
holds for some constant $c_2>0,$ where $\ll=\ll_1-L\e^{\ll_1r_0}>0$
as assumed above. Thus, using H\"older's inequality and applying
Jensen's inequality for the probability measure $\nu(\d s):= \ff{\ll
\e^{-\ll(t-s)}}{1-\e^{-\ll t}}\d s$ on $[0,t]$, we obtain
\begin{equation}\label{SM}\beg{split}
\E\e^{r\|X_t\|_\8^2}&\le \e^{c_3} \big(\E \e^{ c_3 r \|Z_t\|_\infty^2}\big)^{\ff 1 2} \bigg(\E \exp\bigg[ c_3 r\bigg(\ff{(1-\e^{-\ll t})}{\ll} \int_0^t \|Z_s\|_\infty\nu(\d s)\bigg)^2\bigg]\bigg)^{\ff 1 2}\\
&\le \e^{c_3} \big(\E \e^{ c_3 r \|Z_t\|_\infty^2}\big)^{\ff 1 2} \bigg(\int_0^t \E \exp\bigg[ \ff{c_3 r}{\ll^2}  \|Z_s\|_\infty^2\bigg]\nu(\d s)\bigg)^{\ff 1 2}\\
&\le \e^{c_3} \sup_{s\ge 0} \E \exp\Big[\ff{c_3r}{1\land \ll^2} \|Z_s\|_\infty^2\Big],\ \ t\ge 0, r>0\end{split} \end{equation}
for some constant $c_3>0.$ Thus, when $r>0$ is small enough, \eqref{a4}    follows from  \eqref{a5}.
  \end{proof}

Now, we are in position check  condition (iii) in theorem
\ref{T1.1}.

 \begin{lem}\label{inva} Assume   {\bf (A1)} and {\bf (A2)}. If $\ll>0$, then  $P_t$ admits a
unique invariant measure $\mu$. Moreover,  $\mu(\e^{\vv\|\cdot\|_\infty^2})<\infty$ for some  $\vv>0$.
\end{lem}

\begin{proof} The proof is similar to that of \cite[Lemma 2.4]{BWY}.
Let $\mu_t^\xi$ be    the law of $X_t^\xi$. Note that if $\mu_t^\xi$
converges weakly to a probability measure $\mu^\xi$ as $t\to\8,$
then $\mu^\xi$ is an invariant probability measure of $P_t$ (see,
e.g., \cite[Theorem 3.1.1]{DZ}. Let $\scr P(\C)$ be the set of all
probability measures on $\C$. Consider the $L^1$-Wasserstein
distance $W$ induced by $\rr(\xi,\eta):= 1\land
\|\xi-\eta\|_\infty$, i.e.,
$$W(\mu_1,\mu_2):= \inf_{\pi\in \C(\mu_1,\mu_2)} \pi(\rr),\ \ \mu_1,\mu_2\in \scr P(\C), $$ where $ \C(\mu_1,\mu_2)$ is the set of all couplings for $\mu_1$ and $\mu_2.$
 It is well
known that $\scr P(\C)$ is a complete metric space with respect to
the distance $W$ (see, e.g., \cite[Lemma 5.3 and Lemma 5.4]{Chen}),
and the topology induced by $W$ coincides with the weak topology  (see, e.g., \cite[Theorem 5.6]{Chen}). So,
to show existence of an invariant measure, it is sufficient to prove
that $\mu_t^\xi$ is a $W$-Cauchy sequence as $t\to\infty$, i.e.,
\begin{equation}\label{q4}
\lim_{t_1,t_2\to\8} W(\mu_{t_1}^\xi,\mu_{t_2}^\xi)=0.
\end{equation}
For any $t_2>t_1>0$, consider the following SPDEs
\begin{equation*}
\d X(t)=\{AX(t) +b(X_t)\}\d t+\si\d W(t),~t\in[0,t_2],~~X_0=\xi,
\end{equation*}
 and
\begin{equation*}
 \d Y(t)=\{AY(t)+b(Y_t)\}\d t+\si\d
W(t),~t\in[t_2-t_1,t_2],~~Y_{t_2-t_1}=\xi.
\end{equation*}
 Then, the laws
of $X_{t_2}(\xi)$ and $Y_{t_2}(\xi)$ are $\mu_{t_2}^\xi$ and
$\mu_{t_1}^\xi$, respectively. Also,  following an argument leading
to derive \eqref{eq5}, we   obtain
\begin{equation*}
\e^{\ll_1 t}\E\|X_t-Y_t\|^2_\8 \le
c_1\E\|X_{t_2-t_1}-\xi\|^2_\8+L\e^{\ll_1r_0}\int_{t_2-t_1}^t\e^{\ll_1
s}\E\|X_s-Y_s\|^2_\8\d s,~t\in[t_2-t_1,t_2]
\end{equation*}
for some constant $c_1>0$. By   Gronwall's  inequality and
$\ll=\ll_1- L\e^{\ll_1 r_0}>0$ as assumed above, this implies
\begin{equation*}
\E\|X_t-Y_t\|^2_\8 \le c_1\e^{-\ll
(t-t_2+t_1)}\E\|X_{t_2-t_1}-\xi\|^2_\8,~~t\in[t_2-t_1,t_2].
\end{equation*}
Combining this with \eqref{a4} yields
\begin{equation*}
\E\|X_{t_2}-Y_{t_2}\|^2_\8 \le c_2\e^{-\ll t_1}
\end{equation*}
so that
\begin{equation*}
W(\mu_{t_1}^\xi,\mu_{t_2}^\xi)\le \E\|X_{t_2}-Y_{t_2}\|_\8\le
\ss{c_2}\e^{-\ff{\ll t_1}{2}}.
\end{equation*}
Therefore,  \eqref{q4} holds, and, by the completeness of $W$, there
exists $\mu^\xi\in \scr P(\C)$ such that
\begin{equation}\label{w20}
\lim_{t\to\8}W(\mu_t^\xi,\mu^\xi)=0.
\end{equation}

To prove the uniqueness, it suffices to show that $\mu^\xi$ is
independent of $\xi\in\C$. This follows since, by the triangle
inequality, \eqref{eq2} and \eqref{w20},
\begin{equation*}
W(\mu^\xi,\mu^\eta)\le \lim_{t\to\infty}\big\{
W(\mu_t^\xi,\mu^\xi)+W(\mu_t^\eta,\mu^\eta)+W(\mu_t^\xi,\mu_t^\eta)\big\}=0,\ \ \xi,\eta\in \C.
\end{equation*}

  Finally, since  $\mu^0_t\to \mu$ weakly as $t\to\infty$,   by \eqref{a4} we have
 $$\mu(\e^{r\|\cdot\|_\8^2}) =\lim_{N\to\infty} \mu(N\land \e^{r\|\cdot\|_\8^2}) =\lim_{N\to\infty} \lim_{t\to\infty} \E (N\land \e^{r \|X_t^0\|_\8^2})<\infty.$$
 Thus, the proof is finished.
\end{proof}

With the above preparations, we present below a proof of Theorem \ref{T1.2}.

\beg{proof}[Proof of Theorem \ref{T1.2}] According to Theorem
\ref{T1.1}, the first two assertions follow from  \eqref{HH1},
\eqref{eq5} and Lemma \ref{inva}. It remains to prove  the last
assertion. According to \cite[Proposition 2.2]{W10},  the Harnack
inequality \eqref{HH1} implies the log-Harnack inequality
$$P_{t_0} \log f(\xi)\le \log P_{t_0} f(\eta) +\ff {c_0}2 \|\xi-\eta\|_\infty^2,\ \ \ 0<f\in \B_b(\C),~\xi,\eta\in \C.$$
By \cite[Proposition 2.3]{ATW14}, this implies
$$|P_{t_0} f(\xi)-P_{t_0}f(\eta)|^2 \le c_0 \|\xi-\eta\|_\infty^2 \|f\|_\infty^2,\ \ f\in \B_b(\C),~\xi,\eta\in \C.$$
Combining this with the Markov property, we obtain
$$\|\mu_{t_0+t}^\xi- \mu_{t_0+t}^\eta\|_{\mbox{var}}\le 2 \sup_{\|f\|_\infty\le 1} \E|P_{t_0}f(X_t^\xi)- P_{t_0} f(X_t^\eta)|\le 2\ss{c_0}\, \E \|X_t^\xi-X_t^\eta\|_\infty,\ \ t\ge 0.$$
Therefore, the last assertion follows from \eqref{eq2}. \end{proof}

\section{Proof of Theorem \ref{T1.3}}

According to what we have done in the last section for  the proof of Theorem \ref{T1.2}, it suffices to verify the existence and uniqueness of the invariant probability measure, as well as conditions (i)-(iii) in Theorem \ref{T1.1}.
In the present setting we have   to pay more attention on the degenerate part. In particular, the known Harnack inequality (see \cite[Corollary 4.4.4]{W13}) does not meet our requirement
as the exponential term in the upper bound is not integrable with respect to the invariant probability measure. So,
we first  establish the following Harnack inequality which extends the corresponding one in \cite{Wang14} for the case without delay.  The proof is modified from \cite{Wang14} using the coupling by change measures. This method was introduced in \cite{ATW} on manifolds and further developed in \cite{W07} for SPDEs and in \cite{ES} for SDDEs, see
\cite{W13} for a self-contained account on coupling by change of measures and applications.

\begin{lem}\label{Har1}
Assume {\bf (B1)}-{\bf (B4)}. Then, for any $t_0>r_0$, there exists
a constant $c>0$ such that
\begin{equation}\label{w7}
(P_{t_0}f(\bar\xi,\bar\eta))^2 \le \e^{ c(\|\xi-\bar
\xi\|_\8^2+\|\eta-\bar \eta\|_\8^2)} P_{t_0}f^2(\xi,\eta),~~ ~(\xi,\eta),~(\bar\xi,\bar\eta)\in\C,~f\in\B_b(\C).
\end{equation}
\end{lem}

\begin{proof} Let $(X(t),Y(t))=(X^{\xi,\eta}(t), Y^{\xi,\eta}(t))$ for $t\ge 0$, and let $(\bar X(t), \bar Y(t))$ solve the following equation for $(\bar X_0,\bar Y_0)=(\bar\xi,\bar \eta)$:
\begin{equation*}
\begin{cases}
\d \bar X(t)= \{A_1 \bar X(t)+B\bar Y(t)\}\d t,\\
\d \bar Y(t)= \Big\{A_2\bar Y(t)+
b(X_t,Y_t)+\ff{1_{[0,t_0-r_0]}(t)}{t_0-r_0}\e^{tA_2}(\eta(0)-\bar\eta(0))+\e^{tA_2}h'(t)\Big\}\d
t +\si \d W(t),
\end{cases}
\end{equation*}where
\begin{equation}\label{w10}
h(t):=t(t_0-r_0-t)^+B^*\e^{-tA_0^*}e,~~~~t\in[0,t_0]
\end{equation}
for $A_0$ in {\bf (B4)} and  some $e\in \H_1$ to be determined.  Obviously,
\begin{equation}\label{w2}
\bar
Y(t)-Y(t)=\e^{tA_2}\Big\{\ff{(\bar\eta(0)-\eta(0))(t_0-r_0-t)^+}{t_0-r_0}
+h(t)\Big\},~~~~t\in[0,t_0].
\end{equation}
In particular, we have
  $\bar Y_{t_0}=Y_{t_0}$. Next,   the equations of $X(t)$ and $\bar X(t)$ yield
\begin{equation}\label{w1}
\bar X(t)-X(t)=\e^{tA_1}(\bar\xi(0)-\xi(0))+\int_0^t\e^{(t-s)A_1 }B(\bar Y(s)-Y(s))\d s.
\end{equation}  Substituting  \eqref{w2} into \eqref{w1}, we find
that
\begin{equation*}
\bar X(t)-X(t) =\e^{tA_1}  (\bar\xi(0)-\xi(0))
 +\int_0^t\e^{(t-s) A_1}B \e^{sA_2}\Big\{\ff{(\bar\eta(0)-\eta(0))(t_0-r_0-s)^+}{t_0-r_0}  +h(s)\Big\}\d s.
\end{equation*}
By virtue of {\bf (B4)} and the definition of   $h$, this implies
\begin{equation}\label{Hui}
\begin{split}
&\bar X(t)-X(t)\\&=\e^{tA_1 }(\bar\xi(0)-\xi(0))
 +\int_0^t\e^{(t-s)A_1}\e^{sA_1}\e^{sA_0}
B\Big\{\ff{(\bar\eta(0)-\eta(0))(t_0-r_0-s)^+}{t_0-r_0}  +h(s)\Big\}\d s\\
&=\e^{tA_1}\bigg(\bar\xi(0)-\xi(0)
 +\int_0^t\e^{sA_0}B \Big\{\ff{(\bar\eta(0)-\eta(0))(t_0-r_0-s)^+}{t_0-r_0}  +h(s)\Big\}\d s\bigg)\\
&=\e^{tA_1}\bigg(\bar\xi(0)-\xi(0)
 +\int_0^{t_0-r_0}\e^{sA_0}B \Big\{\ff{(\bar\eta(0)-\eta(0))(t_0-r_0-s)}{t_0-r_0}  +h(s)\Big\}\d s\bigg)\end{split}
\end{equation} for any $t\in[t_0-r_0,t_0].$
Moreover,   {\bf (B4)} implies that
\begin{equation*}
\tt
Q_{t_0-r_0}:=\int_0^{t_0-r_0}s(t_0-r_0-s)\e^{sA_0}BB^*\e^{sA_0^*}\d
s
\end{equation*}
is invertible on $\H_1$. In \eqref{w10},  in particular,  take
\begin{equation*}
e=-\tt
Q_{t_0-r_0}^{-1}\Big\{\bar\xi(0)-\xi(0)+\int_0^{t_0-r_0}\ff{t_0-r_0-s}{t_0-r_0}\e^{sA_0}B(\bar\eta(0)-\eta(0))\d
s\Big\}.
\end{equation*}
Then, inserting $h(\cdot)$ back into \eqref{Hui} leads to $\bar
X(t)=X(t)$ for arbitrary $t\in[t_0-r_0,t_0],$ i.e.,
   $\bar X_{t_0}=X_{t_0}$.
Therefore, we arrive at  $ (X_{t_0},Y_{t_0})=(\bar X_{t_0},\bar
Y_{t_0}). $

\smallskip

Let
\begin{equation*}
\tt W(t)=W(t)+\int_0^t\phi(s)\d s,~~~t\in[0,t_0],
\end{equation*}
where
\begin{equation*}
\phi(t):=\si^{-1}\bigg(b(X_t,Y_t)-b(\bar X_t,\bar
Y_t)+\ff{1_{[0,t_0-r_0]}(t)}{t_0-r_0}\e^{tA_2}(\eta(0)-\bar\eta(0))+\e^{tA_2}h'(t)\bigg).
\end{equation*}
By \eqref{w2} and \eqref{Hui}, for some constant $C>0$ we have
\begin{equation}\label{q2}
\|X_t-\bar X_t\|_\8^2+\|Y_t-\bar Y_t\|_\8^2\le C(\|\xi-\bar
\xi\|_\8^2+\|\eta-\bar \eta\|_\8^2),\ \ \ t\in [0,t_0].
\end{equation}
Thus,  by the Girsanov theorem (see, e.g.,
\cite[Theorem 10.14]{DZ}), $\{\tt W(s)\}_{t\in[0,T]}$ is a
cylindrical Wiener process under the weighted probability measure
$\d\Q:= R\d\P$ with
 $$R:= \exp\Big(-\int_0^{t_0}\<\phi(s), \d W(s)\> -\ff 1 2 \int_0^{t_0} |\phi(s)|^2  \d s\Big).$$
Now, we  reformulate the equation for $(\bar X(t), \bar Y(t))$ as
\begin{equation*}
\begin{cases}
\d \bar X(t)= \{ A_1 \bar X(t) +B\bar Y(t)\}\d t, \\
\d \bar Y(t)= \Big\{A_2\bar Y(t)+ b(X_t,Y_t)\Big\}\d t +\si \d \tt
W(t),~~~t\in[0,t_0].
\end{cases}
\end{equation*}
  Then,  invoking  the weak uniqueness of the equation and using $(\bar
X_{t_0} ,\bar Y_{t_0})= (X_{t_0},Y_{t_0})$,    we derive that
\begin{equation*}
\begin{split}
(P_{t_0}f(\bar\xi,\bar\eta))^2&=\big\{\E_\Q
f(\bar X_{t_0}, \bar Y_{t_0})\big\}^2=\big\{\E (Rf(X_{t_0}, Y_{t_0} )\big\}^2\\
 &\le (\E R^2) \E f^2(X_{t_0}, Y_{t_0} )= (\E R^2 )   P_{t_0}f^2(\xi,\eta).
\end{split}
\end{equation*}
Combining this with  \eqref{q2} and the definitions of $R$ and $\phi$, we prove \eqref{w7} for some constant $c>0.$ \end{proof}

Next, the following lemma  verifies   condition (ii) in Theorem \ref{T1.1}. As explained in Section 3 that we may and
 do assume $\ll= \ll_1- \ll' \e^{\ll_1r_0}>0$; otherwise in the sequel it suffices to
 replace $\ll_1$ by $\ll_1'\in (0,\ll_1]$ which attains the maximum in the definition of $\ll$.

\begin{lem}\label{long2} Assume {\bf (B1)}-{\bf (B3)} and let $\eqref{c1}$ hold.
  Then  there
exists  $c>0$ such that for $\ll:= \sup_{s\in (0,\ll_1]} (s -
\ll'\e^{sr_0})>0$,
\begin{equation}\label{w6}\beg{split}
&\|X_t^{\xi,\eta}-X_t^{\bar\xi,
\bar\eta}\|_\8+\|Y_t^{\xi,\eta}-Y_t^{\bar\xi, \bar\eta}\|_\8\le
c(\|\xi-\bar\xi\|_\8+\|\eta-\bar\eta\|_\8)\e^{-\ll t} \end{split}
\end{equation}
for any $t\ge 0, (\xi,\eta), (\bar\xi,\bar\eta)\in \C$.

\end{lem}
\begin{proof}
By {\bf (B1)}-{\bf (B3)}, we have
\begin{equation}\label{WW0} \beg{split}
&\e^{\ll_1t}|X^{\xi,\eta}(t)-X^{\bar\xi,\bar\eta}(t)|-
  |\xi(0)-\bar\xi(0)|\\
  &\le \int_0^t\e^{\ll_1s}\{\dd
|X^{\xi,\eta}(s)-X^{\bar\xi,\bar\eta}(s)|
 +\|B\|\cdot|Y^{\xi,\eta}(s)-Y^{\bar\xi,\bar\eta}(s)|\}\d s,\\
 &\e^{\ll_1t}|Y^{\xi,\eta}(t)-Y^{\bar\xi,\bar\eta}(t)|-
   |\eta(0)-\bar\eta(0)|\\
   &\le \int_0^t\e^{\ll_1s}\{K_1\|X_s^{\xi,\eta}-X_s^{\bar\xi,\bar\eta}\|_\8
 +K_2\|Y_s^{\xi,\eta}-Y_s^{\bar\xi,\bar\eta}\|_\8\}\d s.
\end{split}
\end{equation}
Next, let
\begin{equation}\label{c2}
\aa=\ff{\dd-K_2+\ss{(K_2-\dd)^2+4K_1\|B\|}}{2\|B\|}.
\end{equation}
It is easy to see that $\aa>0$ and, for $\ll'>0$ defined in
\eqref{c1}, we have
\begin{equation}\label{c3}
\aa\dd+K_1=\ll'\aa,\ \ \ \ \aa\|B\|+K_2=\ll'.
\end{equation}
Combining \eqref{WW0}, \eqref{c2}   with \eqref{c3}, we derive
\begin{equation*}
\begin{split}
\e^{\ll_1t}&(\aa\|X_t^{\xi,\eta}-X_t^{\bar\xi,\bar\eta}\|_\8+\|Y_t^{\xi,\eta}-Y_t^{\bar\xi,\bar\eta}\|_\8)\\
&\le\e^{\ll_1r_0}\Big\{\aa\|\xi-\bar\xi\|_\8+\|\eta-\bar\eta\|_\8\\
&\quad+\int_0^t\e^{\ll_1s}((\dd\aa+K_1)\|X_s^{\xi,\eta}-X_s^{\bar\xi,\bar\eta}\|_\8+(\aa\|B\|+K_2)\|Y_s^{\xi,\eta}-Y_s^{\bar\xi,\bar\eta}\|_\8)\d s\Big\}\\
&\le\e^{\ll_1r_0}\Big\{\aa\|\xi-\bar\xi\|_\8+\|\eta-\bar\eta\|_\8\\
&\quad+\ll'\int_0^t\e^{\ll_1s}(\aa\|X_s^{\xi,\eta}-X_s^{\bar\xi,\bar\eta}\|_\8+\|Y_s^{\xi,\eta}-Y_s^{\bar\xi,\bar\eta}\|_\8)\d
s\Big\}.
\end{split}
\end{equation*}
Therefore, we complete the proof by using   Gronwall's inequality
and $\ll= \ll_1- \ll' \e^{\ll_1 r_0}>0$ as assumed above.
\end{proof}

Moreover, corresponding to Lemma \ref{Fer}  for the non-degenerate
case, we have the following result on the exponential integrability
of the segment solution.

\begin{lem}\label{exp1} Assume {\bf (B1)}-{\bf (B3)} and let $\eqref{c1}$ hold. Then there exists a constant $\vv>0$ such that
\begin{equation*}
\sup_{t\ge0}\E\e^{\vv(\|X_t^{\xi,\eta}\|_\8^2+\|Y_t^{\xi,\eta}\|_\8^2)}<\8,\ \ \ (\xi,\eta)\in \C.
\end{equation*}
\end{lem}

\begin{proof} By Lemma \ref{long2}, it suffices to prove for $(\xi,\eta)\equiv(0,0).$ Simply denote $(X_t,Y_t)=(X_t^{0,0}, Y_t^{0,0}).$ We have
\begin{equation*}
X(t)= \int_0^t\e^{(A_1-\dd)(t-s)}\big(\dd X(s)+BY(s)\big)\d s,\ \ t\ge 0.
\end{equation*}
Then,  {\bf (B3)} yields
\begin{equation}\label{H1}
 \e^{\ll_1t}|X(t)|
 \le \int_0^t\e^{\ll_1s}\{\|B\|\cdot|Y(s)| +\dd |X(s)| \}\d
 s.
\end{equation}
Next, according to {\bf (B1)} and {\bf (B2)}, it follows that
\begin{equation}\label{H2}
 \e^{\ll_1t}|Y(t)|
\le \int_0^t\e^{\ll_1s}\{c_0+K_1\|X_s\|_\8+K_2\|Y_s\|_\8\}\d
 s+\e^{\ll_1t}\Big|\int_0^t\e^{A_2 (t-s)}\si\d W(s)\Big|
\end{equation}
holds for  $c_0:=|b(0,0)|.$ Obviously, using $(\H_2, A_2)$ to replace $(\H,A)$, we see that \eqref{a5} holds for
$$Z_t(\theta):= \int_0^{(t+\theta)^+} \e^{A_2 (t-s)}\si\d W(s),\ \ \theta\in [-r_0,0].$$
Combining  \eqref{c3}, \eqref{H1} with  \eqref{H2}, for the present
$Z_t$ we have
\begin{equation*}
\begin{split}
&\e^{\ll_1t}(\aa\|X(t)\|_\8+\|Y(t)\|_\8)\\
&\le\e^{\ll_1r_0}\Big(\aa\sup_{-r_0\le\theta\le0}(\e^{\ll_1(t+\theta)}|X(t+\theta)|)+\sup_{-r_0\le\theta\le0}(\e^{\ll_1(t+\theta)}|Y(t+\theta)|)
\Big)\\
&\le\e^{\ll_1r_0}\Big(\int_0^t\e^{\ll_1s}\{c_0+(\aa\dd+K_1)
\|X_s\|_\8+(\aa\|B\|+K_2)\|Y_s\|_\8\}\d s
 +\e^{\ll_1t}\|Z_{t}\|_\8
\Big)\\
&\le
c_1\e^{\ll_1 t} \big(1+\|Z_{t}\|_\8\big)+\ll'\e^{\ll_1r_0}\int_0^t\e^{\ll_1s}(\aa
\|X_s\|_\8+\|Y_s\|_\8)\d
 s
\end{split}
\end{equation*}
for some constant $c_1>0.$  By Gronwall's inequality and $\ll=
\ll_1- \ll' \e^{\ll_1 r_0}>0$ as assumed above, this yields
\begin{equation*}
\begin{split}
\aa\|X_t\|_\8+\|Y_t\|_\8&\le
c_1\big(1+\|Z_{t} \|_\8 \big)
 +c_1\ll'\e^{\ll_1r_0}\int_0^t\big(1+\|Z_{s}\|_\8\big)
\e^{-\ll''(t-s)}\d s\\
&\le
c_2\bigg( 1+\|Z_{t} \|_\8  +\int_0^t\|Z_{s}\|_\8
\e^{-\ll''(t-s)}\d s\bigg)
\end{split}
\end{equation*}
for some constant $c_2>0$. Hence,   by using H\"older's and Jensen's  inequalities as in \eqref{SM} and applying \eqref{a5} for the present $Z_t$, we
finish the proof. \end{proof}

Finally,  the following lemma ensures the existence and uniqueness of invariant probability measure and verifies condition (iii) in Theorem \ref{T1.1}, so that the proof of Theorem \ref{T1.3} is finished.

\begin{lem}\label{inva'}
Assume {\bf (B1)}- {\bf (B3)} and $\eqref{c1}$. Then $P_t$ has a
unique invariant measure $\mu$. Moreover, $\mu(\e^{\vv \|\cdot\|_\infty^2})<\infty $ holds for some constant $\vv>0.$  \end{lem}

\begin{proof} Let $\mu_t^{\xi,\eta}$ be  the distribution of  $(X_t^{\xi,\eta}, Y_t^{\xi,\eta})$ and let
$$\rr((\xi,\eta), (\bar\xi,\bar\eta))= 1\land (\|\xi-\bar\xi\|_\infty+ \|\eta-\bar\eta\|_\infty).$$
 Making using of Lemmas \ref{long2} and \ref{exp1}  and carrying out an   argument   of Lemma \ref{inva},   we
only need to prove that $\{\mu_t^{\xi,\eta}\}_{t\ge 0}$ is $W$-Cauchy as $t\to\infty$.

For any $t_2>t_1>0$, let $(\tt X(t),\tt Y(t))$ solve    equation \eqref{E1} for $t\in [t_2-t_1,t_2]$ with $(\tt X_{t_2-t_1},\tt Y_{t_2-t_1})= (\xi,\eta).$
 Then,  the laws
of   $(\tt X_{t_2}, \tt Y_{t_2})$ is $\mu_{t_1}^{\xi,\eta}$. So,
\beq\label{CV} W(\mu_{t_1}^{\xi,\eta}, \mu_{t_2}^{\xi,\eta}) \le \E
\big(\|X_{t_2}^{\xi,\eta}-\tt X_{t_2}\|_\infty + \|
Y_{t_2}^{\xi,\eta}-\tt Y_{t_2}\|_\infty\big). \end{equation} Next,
repeating the proof of  Lemma \ref{long2}  for  $t\in [t_2-t_1,t_2]$
and  $(\tt X_t, \tt Y_t)$ in place of $(X_t^{\bar \xi,\bar\eta},
Y_t^{\bar\xi,\bar\eta})$, we obtain
$$\|X_{t_2}^{\xi,\eta}-\tt X_{t_2}\|_\infty + \| Y_{t_2}^{\xi,\eta}-\tt Y_{t_2}\|_\infty \le c(\|\xi-X_{t_2-t_1}^{\xi,\eta}\|_\8+
\|\eta-Y_{t_2-t_1}^{\xi,\eta}\|_\8)\e^{-\ll t_1}$$
 for some constant $c>0$ independent of $t_1$ and $t_2.$ Combining this with \eqref{CV} and using Lemma \ref{exp1}, we prove
 $\lim_{t_1,t_2\to \infty} W(\mu_{t_1}^{\xi,\eta}, \mu_{t_2}^{\xi,\eta}) =0.$ \end{proof}

\noindent{\bf Acknowledgment }

\smallskip

We are indebted to the referee for his/her valuable comments which
have greatly improved our earlier version.

\end{document}